\documentclass{amsart}
\usepackage{tikz}
\usepackage{amssymb,amsbsy}
\usepackage{amsrefs}
\newtheorem{theorem}{Theorem}[section]
\newtheorem{lemma}[theorem]{Lemma}
\newtheorem{proposition}[theorem]{Proposition}
\newtheorem{corollary}[theorem]{Corollary}
\theoremstyle{definition}
\newtheorem{definition}[theorem]{Definition}
\newtheorem{example}[theorem]{Example}
\newcommand{\field}{\boldsymbol{k}}
\newcommand{\Prim}{\mathop{\mathrm{Prim}}}
\newcommand{\Z}{\mathbb Z}
\newcommand{\1}{_{(1)}}
\newcommand{\2}{_{(2)}}
\newcommand{\3}{_{(3)}}
\DeclareMathOperator{\ad}{\mathrm{ad}}
\DeclareMathOperator{\BCH}{\mathrm{BCH}}
\DeclareMathOperator{\Endo}{\mathrm{End}}
\DeclareMathOperator{\Id}{\mathrm{Id}}
\DeclareMathOperator{\Sab}{\mathrm{Sab}}
\begin{document}
\title{A Non-associative Baker-Campbell-Hausdorff formula}%

\author{J. Mostovoy}
\address{Departamento de Matem\'aticas, CINVESTAV-IPN, Apartado Postal 14--740, 07000 M\'e\-xico D.F., Mexico}
\email{jacob@math.cinvestav.mx}

\author{J. M. P\'erez-Izquierdo}
\address{Departamento de Matem\'aticas y Computaci\'on, Universidad de
La Rioja, 26004 \\ Lo\-gro\-\~no, Spain}
\email{jm.perez@unirioja.es}
\thanks{The authors acknowledge the support by the Spanish Ministerio de Ciencia e Innovaci\'on (MTM2013-45588-C3-3-P); J.\ Mostovoy was also supported by the CONACYT grant 168093-F; I. P. Shestakov also acknowledges support by FAPESP, processo 2014/09310-5 and CNPq, processos 303916/2014-1 and 456698/2014-0.}%

\author{I. P. Shestakov}
\address{Instituto de Matem\'atica e Estat{\'\i}stica, Universidade de S\~ao Paulo, Caixa Postal 66281, S\~ao Paulo, SP 05311-970, Brazil}
\email{shestak@ime.usp.br}

\keywords{Baker-Campbell-Hausdorff formula, primitive elements, Sabinin algebras, Magnus expansion}
\subjclass[2010]{17A50,20N05}

%%%%%%%%%%%%%%%%%%%%%%%%%%%%%%%%%%%%%%%%%%%%%%%%%
%    _   _         _                  _   
%   /_\ | |__  ___| |_ _ __ __ _  ___| |_ 
%  //_\\| '_ \/ __| __| '__/ _` |/ __| __|
% /  _  \ |_) \__ \ |_| | | (_| | (__| |_ 
% \_/ \_/_.__/|___/\__|_|  \__,_|\___|\__|
%
\begin{abstract}
We address the problem of constructing the non-associative version of the Dynkin form of the Baker-Campbell-Hausdorff formula; that is, expressing $\log (\exp (x)\exp(y))$, where $x$ and $y$ are non-associative variables, in terms of the Shestakov-Umirbaev primitive operations. In particular, we obtain a recursive expression for the Magnus expansion of the  Baker-Campbell-Hausdorff series and an explicit formula in degrees smaller than 5. Our main tool is a non-associative version of the Dynkin-Specht-Wever Lemma. A construction of Bernouilli numbers in terms of binary trees is also recovered.
\end{abstract}
%%%%%%%%%%%%%%%%%%%%%%%%%%%%%%%%%%%%%%%%%%%%%%%%%
\maketitle
%-. Se conoce mucho más de la estructura de las álgebras de Hopf no asociativas y del producto de los primitivos. Es hora de establecer una fórmula BCH.
%
%
%-. Aclararse con si las operaciones p en algebras de Hopf se aplican a elementos compuestos o a que. Clarificar el uso de p(u;v;a) en algebras de Hopf arbitrarias.
%
%-. Aclarar la idea de simplificar  usando $\backslash$
%
%-. Aclara que (uv)p(u,v,a) = (u,v,a)
%
%-. Introducir la derivacion por grado en las graduadas
%
%%%%%%%%%%%%%%%%%%%%%%%%%%%%%%%%%%%%%%%%%%%%%%%%%%%%%%%%%%%%%%%%%%
%   _____       _                 _            _   _             
%   \_   \_ __ | |_ _ __ ___   __| |_   _  ___| |_(_) ___  _ __  
%    / /\/ '_ \| __| '__/ _ \ / _` | | | |/ __| __| |/ _ \| '_ \ 
% /\/ /_ | | | | |_| | | (_) | (_| | |_| | (__| |_| | (_) | | | |
% \____/ |_| |_|\__|_|  \___/ \__,_|\__,_|\___|\__|_|\___/|_| |_|
%
\section{Introduction}
%% Estructrua
%%%%%%%%%%%%%%%%%%%%%%%%%%%%%%%%%%%%%%%%%%%%%%%%%%%%%%%%%%%%%%%%%%
The Baker-Campbell-Hausdorff formula is the expansion of $\log(\exp(x)\exp(y))$ in terms of nested commutators for the non-commuting variables $x$ and $y$, where the commutator of $a$ and $b$ is defined as $[a,b] := ab - ba$.  The explicit combinatorial form of it was given by Dynkin in his 1947 paper \cite{Dy47}. By considering the linear extension of the map $\gamma$ defined by $\gamma(1) := 0$, $\gamma(x) := x$, $\gamma(y) := y$, $\gamma(ux) := [\gamma(u),x]$ and $\gamma(uy) := [\gamma(u),y]$ he proved that
\begin{multline}\label{eq:Dynkin_formula}
\log(\exp(x)\exp(y)) = \\ \sum_{n=1}^{\infty}\frac{(-1)^{n-1}}{n}\sum_{
	r_i+s_i \geq 1} \frac{(\sum_{j=1}^{n}(r_j+s_j))^{-1}}{r_1!s_1!\cdots r_n!s_n!}\gamma(x^{r_1}y^{s_1}\cdots x^{r_n}y^{s_n}).
\end{multline}
This series is related to Lie's Third Theorem and the history around it is too rich to be retold here, so we refer the reader to the recent monograph \cite{BonFul12} and references therein for a historical account.

The Baker-Campbell-Hausdorff formula, as well as many other results in Lie theory, firmly belongs to associative algebra. However, after the work of Mikheev and Sabinin on local analytic loops \cite{SaMi87} and the description of the primitive operations in non-associative algebras by Shestakov and Umirbaev \cite{ShUm02}, associativity does not seem to be as essential for the Lie theory as previously thought \cite{MPS14}.

In this paper we address the problem of determining $\BCH_l(x,y)$ in
\begin{displaymath}
	\exp_l(x)\exp_l(y) = \exp_l(\BCH_l(x,y)),
\end{displaymath}
where 
\begin{displaymath}
		\exp_l(x) = \sum_{n\geq 0} \frac{1}{n!}\underbrace{(((xx)\cdots )x)x}_{n},
\end{displaymath}
in terms of Shestakov-Umirbaev operations for the primitive elements of the non-associative algebra freely generated by $x$ and $y$. Our approach uses a generalization of the Magnus expansion (see \cite{BlaCaOtRo09} for a readable survey), that is, we will study the differential equation $$X'(t) = X(t)A(t),$$ where $X(t)$ stands for $\exp_l(\Omega(t))$ and both $A(t)$ and $\Omega(t)$ belong to a non-associative algebra. The differential equation $$\Omega'(t)= A(t) + \sum_{J} n_J P_J(\Omega(t);A(t))$$ satisfied by $\Omega(t)$ (Corollary~\ref{cor:Magnus}) is obtained with the help of a non-associative version of the Dynkin-Specht-Wever Lemma (Lemma~\ref{lem:DSW}). This equation leads to a recursive formula for computing the expansion of $\BCH_l(x,y)$, which gives, in degrees smaller than 5, the following expression:
\begin{align*}
		\BCH_l(x,y) &=  x+y + \frac{1}{2}[x,y]\\
						& \quad +\frac{1}{12}[x,[x,y]] - \frac{1}{3}\langle x;x,y\rangle 
						-\frac{1}{12} [y,[x,y]] - \frac{1}{6} \langle y; x,y \rangle -\frac{1}{2}\Phi(x;y,y) \\
						&\quad -\frac{1}{24}\langle x; x, [x,y]\rangle - \frac{1}{12} [x,\langle x;x,y\rangle] - \frac{1}{8}\langle x,x;x,y\rangle \\
						&\quad +\frac{1}{24}[[x,[x,y]],y] - \frac{1}{24}[x,\langle y;x,y\rangle] -\frac{1}{4}\Phi(x,x;y,y) - \frac{1}{4}[x,\Phi(x;y,y)] \\
						&\quad -\frac{1}{24}[\langle x;x,y\rangle,y] -\frac{1}{24}\langle x; [x,y],y\rangle - \frac{1}{6}\langle x,y;x,y\rangle + \frac{1}{24}\langle y,x;x,y\rangle \\
						& \quad + \frac{1}{12} [\Phi(x;y,y),y] + \frac{1}{24}\langle y; y,[x,y]\rangle - \frac{1}{24} \langle y,y;x,y\rangle - \frac{1}{6}\Phi(x;y,y,y) \\
						& \quad + \dots
\end{align*}
When all the operations apart from $[\, ,\,]$ vanish we recover the usual Baker-Campbell-Hausdorff formula. A different approach to the non-associative Baker-Campbell-Hausdorff formula has appeared in \cite{GeHo03}; it does not explicitly use the Dynkin-Specht-Wever lemma or the Magnus expansion. For the treatment of the subject from the point of view of differential geometry see \cite{GW}; actually, geometric considerations also motivate a different type of a Baker-Campbell-Hausdorff formula, see \cite{MPS14b}; although it is of importance for the non-associative Lie theory, we shall not consider it here.

Our results are presented for the unital $\field$-algebra of  formal power series $\field\{\{x,y\}\}$ in two non-associative variables $x$ and $y$. Readers with background in non-associa\-tive structures will realize that a more natural context for the Baker-Campbell-Hausdorff formula is the completion of the universal enveloping algebra of a relatively free Sabinin algebra on two generators. The extension of our results to that context is rather straightforward.

Readers familiar with free Lie algebras might wonder about the existence, behind of the scenes, of certain non-associative Lie idempotents responsible for some of our formulas. The answer is affirmative; however, this topic is not discussed in the present paper since it requires some knowledge of Sabinin algebras and treating it would significantly increase the length of text. Very briefly, the context for the non-associative Lie idempotents is as follows.

One can start with a variety $\Omega$ of loops containing all the abelian groups and define a relatively free Sabinin algebra $\Sab_\Omega(X)$ associated to the variety $\Omega$ and  freely generated by $X:=\{ x_1,x_2,\dots\}$. Let $U(\Sab_\Omega(X))$ be the universal enveloping algebra of $\Sab_\Omega(X)$. This algebra is a  non-associative graded Hopf algebra $$U(\Sab_\Omega(X)) = \bigoplus_{n=0}^\infty U_n$$ once we set $\vert x_i \vert := i$. The convolution $$(f * g)(x) = \sum f(x\1) g(x\2)$$ defines a non-associative product on the space of $\field$-linear maps $\Endo(U(\Sab_\Omega(X)))$. The subalgebra generated by the projections $\Id_n$ of $U(\Sab_\Omega(X))$ onto $U_n$ (where $n = 0, 1,\dots$) with respect to this convolution product, which in the associative setting is anti-isomorphic to Solomon's descent algebra, is isomorphic as a graded Hopf algebra to $U(\Sab_\Omega(X))$. Therefore, $U(\Sab_\Omega(X))$ is a non-associative Hopf algebra with an extra associative (inner) product inherited from the composition in $\Endo(U(\Sab_\Omega(X)))$. In the associative case, this is the subject of study in the theory of non-commutative symmetric functions, so a similar theory seems possible in the non-associative setting. Eulerian, Dynkin and Klyachko idempotents, among others, are easily understood in $U(\Sab_\Omega(X))$ as particular examples of primitive elements with respect to the comultiplication that, in addition, are idempotent with respect to the associative inner product, and they ultimately explain some of the formulas in this paper.

%%%%%%%%%%%%%%%%%%%%%%%%%%%%%%%%%%%%%%%%%%%%%
%    _   __      __        __  _           
%   / | / /___  / /_____ _/ /_(_)___  ____ 
%  /  |/ / __ \/ __/ __ `/ __/ / __ \/ __ \
% / /|  / /_/ / /_/ /_/ / /_/ / /_/ / / / /
%/_/ |_/\____/\__/\__,_/\__/_/\____/_/ /_/ 
%                                          
\subsection{Notation.} 
%%%%%%%%%%%%%%%%%%%%%%%%%%%%%%%%%%%%%%%%%%%%%
Throughout this paper the characteristic of the base field $\field$ is zero. The unital associative $\field$-algebra freely generated by a set of generators $X$ will be denoted by $\field \langle X \rangle$ while $\field \langle\langle X \rangle\rangle$ will stand for the unital  associative algebra of formal power series on $X$ with coefficients in $\field$. Their non-associative counterparts, namely, the unital non-associative $\field$-algebra freely generated by $X$ and the unital non-associative $\field$-algebra of  formal power series on $X$ with coefficients in $\field$, will be denoted by $\field \{ X \}$ and $\field\{\{X\}\}$ respectively. For any algebra $H$, $H[[t]]$ will denote the algebra of formal power series in $t$ with coefficients in $H$. The parameter $t$ commutes and associates with all the elements in $H[[t]]$. Finally, we will stick to the following order of parentheses for powers: $x^n := (((xx)\cdots )x)x$ ($n$ times).

%%%%%%%%%%%%%%%%%%%%%%%%%%%%%%%%%%%%%%%%%%%%%%%%%%%%%%%%%%%%%%%%%%%%%%
%   ___                _                            _        _     
%  / __\   _ _ __   __| | __ _ _ __ ___   ___ _ __ | |_ __ _| |___ 
% / _\| | | | '_ \ / _` |/ _` | '_ ` _ \ / _ \ '_ \| __/ _` | / __|
%/ /  | |_| | | | | (_| | (_| | | | | | |  __/ | | | || (_| | \__ \
%\/    \__,_|_| |_|\__,_|\__,_|_| |_| |_|\___|_| |_|\__\__,_|_|___/
%                                                                  
\section{Fundamentals}

%%%%%%%%%%%%%%%%%%%%%%%%%%%%%%%%%%%%%%%%%%%%%%%%%%%%%%%%%%%%%%%%%%%%%%%
%    __  __            ____         __           __                   
%   / / / /___  ____  / __/  ____ _/ /___ ____  / /_  _________ ______
%  / /_/ / __ \/ __ \/ /_   / __ `/ / __ `/ _ \/ __ \/ ___/ __ `/ ___/
% / __  / /_/ / /_/ / __/  / /_/ / / /_/ /  __/ /_/ / /  / /_/ (__  ) 
%/_/ /_/\____/ .___/_/     \__,_/_/\__, /\___/_.___/_/   \__,_/____/  
%           /_/                   /____/                              
\subsection{Non-associative Hopf algebras.}
%%%%%%%%%%%%%%%%%%%%%%%%%%%%%%%%%%%%%%%%%%%%%%%%%%%%%%%%%%%%%%%%%%%%%%%
A coalgebra $(C, \Delta, \epsilon)$ is a vector space equipped with two linear maps  $\Delta \colon C \rightarrow C \otimes C$ (\emph{comultiplicaton}) and   $\epsilon \colon C \rightarrow \field$ (\emph{counit}) such that
\begin{displaymath}
	\sum \epsilon(x\1) x\2 = x = \sum \epsilon(x\2)x\1,
\end{displaymath}
where $\sum x\1 \otimes x\2$ stands for $\Delta(x)$ (\emph{Sweedler notation}). 
Coassociative and cocommutative coalgebras are those coalgebras $(C, \Delta, \epsilon)$ that, in addition, satisfy $$(\Delta \otimes \Id) \Delta = (\Id \otimes \Delta) \Delta$$ (\emph{coassociativity}) and $\tau \Delta = \Delta$ (\emph{cocommutativity}) where $\tau (x \otimes y) = y \otimes x$.
Coassociativity ensures that $$\sum {x\1}_{(1)} \otimes {x\1}_{(2)} \otimes x\2 = \sum x\1 \otimes {x\2}_{(1)} \otimes {x\2}_{(2)}$$ so we can safely write $\sum x\1 \otimes x\2 \otimes x\3$ for any of the sides of this equality. 
For coassociative coalgebras the result of the iterated application $n$ times of $\Delta$  to $x$ does not depend on the selected factors and it is denoted by $\sum x\1 \otimes x\2 \otimes \cdots \otimes x_{(n+1)}$. 
Cocommutativity ensures that we can freely permute the factors of $\sum x\1 \otimes x\2 \otimes \cdots \otimes x_{(n+1)}$ without altering the value of this expression \cite{Sw69}.

In this paper, by a \emph{(non-associative) Hopf algebra} $(H, m, u, \backslash, /, \Delta, \epsilon)$ we shall mean a cocommutative and coassociative coalgebra $(H, \Delta, \epsilon)$ endowed with the following linear maps: a \emph{product} $m \colon H \otimes H \rightarrow H$, a \emph{unit} $u \colon \field \rightarrow H$, a \emph{left division} $\backslash \colon H \otimes H \rightarrow H$ and a \emph{right division} $/\colon H \otimes H \rightarrow H$ so that $\Delta(xy) = \Delta(x) \Delta(y)$, $\Delta(1) = 1 \otimes 1$, $\epsilon(xy) = \epsilon(x) \epsilon(y)$,  $\epsilon(1) = 1$ and
\begin{align}
 	\sum x\1 \backslash (x\2 y) &= \epsilon(x) y = \sum x\1 (x\2 \backslash y) \label{eq:left_division}\\
 	\sum (y x\1) / x\2 &= \epsilon(x) y = \sum (y / x\1) x\2 \label{eq:right_division}
\end{align}
where $xy := m(x \otimes y)$ and $1 := u(1)$ is the unit element (see \cite{MPS14} for a survey on non-associative Hopf algebras). 
In case that $H$ is associative then the left and right divisions can be written as $x \backslash y = S(x) y$ and $x/y = xS(y)$ where $S$ is the antipode. However, non-associative Hopf algebras lack antipodes in general. 
 
%%%%%%%%%%%%%%%%%%%%%%%%%%%%%%%%%%%%%%%%%%%%%%%%%%%%%%%%%%%%%%%%%%%%%%%
%    ______                        __           __                   
%   / ____/_______  ___     ____ _/ /___ ____  / /_  _________ ______
%  / /_  / ___/ _ \/ _ \   / __ `/ / __ `/ _ \/ __ \/ ___/ __ `/ ___/
% / __/ / /  /  __/  __/  / /_/ / / /_/ /  __/ /_/ / /  / /_/ (__  ) 
%/_/   /_/   \___/\___/   \__,_/_/\__, /\___/_.___/_/   \__,_/____/  
%                                /____/                              
\subsection{The free unital non-associative algebra $\field\{X\}$.} 
%%%%%%%%%%%%%%%%%%%%%%%%%%%%%%%%%%%%%%%%%%%%%%%%%%%%%%%%%%%%%%%%%%%%%%%%
The most important example of a non-associative Hopf algebra in this paper is the unital non-associative algebra $\field\{X\}$ freely generated by $X:= \{x_1,x_2,\dots \}$. 
The maps $\Delta \colon x_i \mapsto x_i \otimes 1 + 1 \otimes x_i$ and $\epsilon \colon x_i \mapsto 0$ ($i = 1,2,\dots$) induce homomorphisms of unital algebras $\Delta \colon \field\{X\} \rightarrow \field\{X\} \otimes \field\{X\}$ and $\epsilon \colon \field\{X\} \rightarrow \field$ so that $(\field\{X\}, \Delta, \epsilon)$ is a coassociative and cocommutative coalgebra. 
By induction on the degree of $x$,  the formulas (\ref{eq:left_division}) and (\ref{eq:right_division}) uniquely determine the left and the right division in $\field\{X\}$. 
For instance, $1 \backslash (1 y) = \epsilon(1) y$ implies $1\backslash y = y$ and
\begin{displaymath}
	x_i \backslash (1y) + 1 \backslash (x_i y) = \epsilon(x_i) y = 0 \quad \text{implies}\quad  x_i \backslash y = -x_i y
\end{displaymath}
etc. 
The operations $\Delta$, $\epsilon$, $\backslash$ and $/$, together with the product and the unit, provide $\field\{X\}$ with the  structure of a non-associative Hopf algebra. 
Far from being a fancy feature, the divisions are a valuable tool for computations.
%%%%%%%%%%%%%%%%%%%%%%%%%%%%%%%%%%%%%%%%%%%%%%%%%%%%%%%%%%%%%%%%%%%%%%%%%%%%%%%%%%%%%%%%%%%%%
%    ____       _           _ __  _                    __                          __      
%   / __ \_____(_)___ ___  (_) /_(_)   _____     ___  / /__  ____ ___  ___  ____  / /______
%  / /_/ / ___/ / __ `__ \/ / __/ / | / / _ \   / _ \/ / _ \/ __ `__ \/ _ \/ __ \/ __/ ___/
% / ____/ /  / / / / / / / / /_/ /| |/ /  __/  /  __/ /  __/ / / / / /  __/ / / / /_(__  ) 
%/_/   /_/  /_/_/ /_/ /_/_/\__/_/ |___/\___/   \___/_/\___/_/ /_/ /_/\___/_/ /_/\__/____/  
%                                                                                          
\subsection{Primitive elements of $\field\{X\}$ and the Shestakov-Umirbaev operations.}
%%%%%%%%%%%%%%%%%%%%%%%%%%%%%%%%%%%%%%%%%%%%%%%%%%%%%%%%%%%%%%%%%%%%%%%%%%%%%%%%%%%%%%%%%%%%%
An element $a$ in a Hopf algebra $H$ such that 
\begin{displaymath}
	\Delta(a) = a \otimes 1 + 1 \otimes a
\end{displaymath}
is called \emph{primitive}; the subspace of all such elements is denoted by $\Prim(H)$. 
While for associative Hopf algebras this subspace is a Lie algebra with the commutator product $[x,y]:= xy - yx$, Shestakov and Umirbaev \cite{ShUm02} realized that if $H$ is non-associative, many more operations are required to  describe its algebraic structure completely. 

Let $X:= \{x, x_1,x_2,\dots\}$, $Y:= \{y, y_1,y_2,\dots\}$ and $Z:=\{z\}$ be disjoint sets of symbols that we take to be the free generators of $\field\{X \cup Y \cup Z\}$. 
Write $\underline{x} := ((x_1x_2)\cdots)x_m$, $\underline{y}:= ((y_1y_1)\cdots) y_n$ and define 
\begin{equation}\label{eq:p}
	p(x_1,\dots, x_m; y_1,\dots,y_n;z) := p(\underline{x},\underline{y},z) := \sum (\underline{x}\1 \underline{y}\1) \backslash (\underline{x}\2, \underline{y}\2,z)
\end{equation}
in $\field\{X \cup Y \cup Z\}$, where $(x,y,z)$ denotes the associator $(xy)z - x(yz)$ of $x,y$ and $z$. Each of the elements $p(\underline{x},\underline{y},z)$ is primitive.
Considered as non-associative polynomials, $p(x_1,\dots, x_m;y_1,\dots,y_n;z)$ can be evaluated in any algebra $A$ so we can think of them as of new multilinear operations derived from the binary product of $A$. Define
\begin{align*}
	[x,y] &:= xy - yx \\
	\langle x_1,\dots,x_m; y,z \rangle &:= -p(x_1,\dots, x_m;y;z) + p(x_1,\dots,x_m; z;y)\\
	\Phi(x_1,\dots, x_m; y_1,\dots, y_n; y_{n+1}) &:= \\
	& \hskip -2.5cm \frac{1}{m!(n+1)!} \sum_{\sigma \in S_n, \tau \in S_{m+1}} p(x_{\sigma(1)},\dots, x_{\sigma(m)}; y_{\tau(1)},\dots, y_{\tau(n)}; y_{\tau(n+1)})
\end{align*}
where $m, n \geq 1$ and $S_k$ stands for the symmetric group on $\{1,\dots, k\}$.
In order to simplify the notation, for $m = 0$ we write $$\langle y,z\rangle := \langle x_1,\dots,x_m; y,z \rangle:= \langle 1 ; y,z \rangle:= -[y,z].$$
With this convention, (\ref{eq:p}) gives
\begin{equation}\label{eq:brackets}
	(\underline{x}y)z - (\underline{x}z)y = - \sum \underline{x}_{(1)} \langle \underline{x}_{(2)}; y,z \rangle.
\end{equation}
Shestakov and Umirbaev proved that
\begin{displaymath}
	(\Prim(\field\{X\}),\langle \,\, ; \, , \, \rangle,\Phi(\,;\,;\,)) \text{ is generated by } X.
\end{displaymath}
Thus, while (\ref{eq:Dynkin_formula}) can be written in terms of commutators, the natural language to write its non-associative counterpart uses $\langle \,\, ; \, , \, \rangle$ and $\Phi(\,;\,;\,)$. 
%%%%%%%%%%%%%%%%%%%%%%%%%%%%%%%%%%%%%%%%%%%%%%%%%%%%%%%%%%%%%%%%%%%%%%%%%%%%%%%%%
%    ______                                  __  _       __    
%   / ____/  ______  ____  ____  ___  ____  / /_(_)___ _/ /____
%  / __/ | |/_/ __ \/ __ \/ __ \/ _ \/ __ \/ __/ / __ `/ / ___/
% / /____>  </ /_/ / /_/ / / / /  __/ / / / /_/ / /_/ / (__  ) 
%/_____/_/|_/ .___/\____/_/ /_/\___/_/ /_/\__/_/\__,_/_/____/  
%          /_/                                                 
\subsection{Exponentials, logarithms and the Baker-Campbell-Hausdorff formula.}
%%%%%%%%%%%%%%%%%%%%%%%%%%%%%%%%%%%%%%%%%%%%%%%%%%%%%%%%%%%%%%%%%%%%%%%%%%%%%%%%%

The algebra $\field\{\{ x\}\}$ (respectively, $\field\langle\langle x \rangle\rangle$) of formal power series in $x$ with coefficients in $\field$ is a topological Hopf algebra with the continuous extension of the operations of the Hopf algebra $\field\{x\}$ (respectively,  $\field\langle x \rangle$).
Since $\Prim(\field \langle x \rangle) = \field x$, the \emph{group-like} elements of  $\field\langle\langle x \rangle\rangle$, that is, the elements $g$ such that $\Delta(g) = g \otimes g$ and $\epsilon(g) = 1$, are of the form $\exp(\alpha x)$ with $\alpha \in \field$. 
Therefore, $\exp(x)$ is, in a sense, canonical among all of them. 
However, $\Prim(\field\{\{x\}\})$ is infinite-dimensional and $\field\{\{x\}\}$ has an infinite number of group-like elements that could rightfully be considered as the non-associative analogs of the exponential series. 
Apart from the most obvious non-associative versions of the exponential
\begin{displaymath}
	\exp_l(x) := \sum_{n\geq 0} \frac{1}{n!}\underbrace{(((xx)\cdots )x)x}_{n} \quad\text{and}\quad\exp_r(x) := \sum_{n\geq 0} \frac{1}{n!}\underbrace{x(x(\cdots(xx)))}_{n}
\end{displaymath}
other series have been proposed as non-associative analogs of  $\exp(x) := \sum_{n=0}^\infty {x^n}/{n!}$, each leading to a different logarithm \cite{Ge04a}.

%%%
\begin{definition}
	A group-like element $e(x) \in \field\{\{x\}\}$ is a \emph{base for logarithms} if its homogeneous component $e_1(x)$ of degree one in $x$ is not zero. We say that the base for logarithms $e(x)$ is \emph{normalized} if $e_1(x) = x$. 
\end{definition}
%%%
Associated with any base for logarithms $e(x)$ there exists a primitive element $\log_e(x) \in \field\{\{x\}\}$ determined by
\begin{displaymath}
	e(\log_e(x)) = x = \log_e(e(x)).
\end{displaymath} 
The \emph{exponentiation} on $\field\{\{X\}\}$ \emph{with base} $e(x)$ and the \emph{logarithm on $ \field\{\{X\}\}$ to the base $e(x)$} are the maps 
\begin{align*}
	e\colon \field\{\{X\}\}_+ &\rightarrow 1+\field\{\{X\}\}_+ & \log_{e} \colon 1 +  \field\{\{X\}\}_+ &\rightarrow  \field\{\{X\}\}_+\\
u&\mapsto e(u) & 1+ u &\mapsto \log_e(1+ u)
\end{align*} 
where $ \field\{\{X\}\}_+$ denotes the space of formal power series with zero constant term. Both maps are inverse to each other and give a bijection between the primitive and the group-like elements in $ \field\{\{X\}\}$. 
The logarithms to the bases $\exp_l(x)$ and $\exp_r(x)$ will be denoted by $\log_l$ and $\log_r$, respectively.

Any base for logarithms $e(x)$ determines a Baker-Campbell-Hausdorff series in $\field \{\{x,y\}\}$:
\begin{displaymath}
	\BCH_e(x,y):= \log_e(e(x)e(y)).
\end{displaymath}
The element $\BCH_e(x,y)$ is primitive so it can be written in terms of the Shestakov-Umirbaev operations $\langle \,\, ; \, , \, \rangle$ and $\Phi(\,;\,)$. 
Since these operations are defined via the left-normed products $\underline{x}$ and $\underline{y}$, the base $\exp_l(x)$ is better adapted to recursive computations.  In \cite{MP10} $\log_l(1+x)$ has been described as follows. For $\tau=x$ set $B_\tau :=\tau! :=1$. If $\tau\neq x$ is a non-associative monomial in $x$, there is only one way of writing $\tau$ as a product $(\ldots((x
\tau_1)\tau_2)\ldots)\tau_k$. Set $B_{\tau}:=B_k B_{\tau_1}\ldots B_{\tau_k}$ and $\tau! := k! \tau_1 !\ldots \tau_k !$ where $B_k$ is the $k$th Bernoulli number. With this notation we have
\begin{displaymath}
	\log_l(1+x)=\sum_{\tau} \frac{B_\tau}{\tau!} \tau \in \field\{\{x\}\}.
\end{displaymath}
The Baker-Campbell-Hausdorff series for different bases are related in a straightforward manner. If $e$ and $f$ are two bases for logarithms, the series $h(x):=\log_{f}(e(x))$ is a primitive element of $\field\{\{x\}\}$ whose term of degree 1 is non-zero. In particular, it has a composition inverse $h^{-1}(x)=\log_e(f(x))$ such that $h^{-1}(h(x))=x$. It is then clear that
$$\BCH_e(x,y)=h^{-1}(\BCH_f(h(x),h(y))).$$
Moreover, for any Baker-Campbell-Hausdorff series $\BCH(x,y)$ and any primitive $h\in\field\{\{x\}\}$ with $h_1\neq 0$, the series $h^{-1}(\BCH(h(x),h(y)))$ is also a Baker-Campbell-Hausdorff series for some base.

%%%%%%%%%%%%%%%%%%%%%%%%%%%%%%%%%%%%%%%%%%%%%%%%%%%%%%%%%%%%%%%
%   ___    ___            ___                          _       
%  / __\  / __\ /\  /\   / __\__  _ __ _ __ ___  _   _| | __ _ 
% /__\// / /   / /_/ /  / _\/ _ \| '__| '_ ` _ \| | | | |/ _` |
%/ \/  \/ /___/ __  /  / / | (_) | |  | | | | | | |_| | | (_| |
%\_____/\____/\/ /_/   \/   \___/|_|  |_| |_| |_|\__,_|_|\__,_|
%                                                              
\section{A Nonassociative Baker-Campbell-Hausdorff formula}
%%%%%%%%%%%%%%%%%%%%%%%%%%%%%%%%%%%%%%%%%%%%%%%%%%%%%%%%%%%%%%%%%%%%%%%%%%%%%%%%%%%%%%%%%%%%%%%%%%%%%%%%%%
%    ____              __   _            _____            __       __        _       __                   
%   / __ \__  ______  / /__(_)___       / ___/____  ___  / /______/ /_      | |     / /__ _   _____  _____
%  / / / / / / / __ \/ //_/ / __ \______\__ \/ __ \/ _ \/ __/ ___/ __ \_____| | /| / / _ \ | / / _ \/ ___/
% / /_/ / /_/ / / / / ,< / / / / /_____/__/ / /_/ /  __/ /_/ /__/ / / /_____/ |/ |/ /  __/ |/ /  __/ /    
%/_____/\__, /_/ /_/_/|_/_/_/ /_/     /____/ .___/\___/\__/\___/_/ /_/      |__/|__/\___/|___/\___/_/     
%      /____/                             /_/                                                             
\subsection{A non-associative Dynkin-Specht-Wever Lemma.} 
%%%%%%%%%%%%%%%%%%%%%%%%%%%%%%%%%%%%%%%%%%%%%%%%%%%%%%%%%%%%%%%%%%%%%%%%%%%%%%%%%%%%%%%%%%%%%%%%%%%%%%%%%
Let $d$ be a derivation of $\field\{X\}$ that preserves $\Prim(\field\{X\})$, that is
\begin{displaymath}
	d(\Prim(\field\{X\})) \subseteq \Prim(\field\{X\}).
\end{displaymath}
Define $\gamma_d(u) := \sum u_{(1)} \backslash d(u_{(2)})$; thus,
\begin{displaymath}
	d(u) = \sum u_{(1)} \gamma_d(u_{(2)})
\end{displaymath}
for all $u \in \field\{X\}$. The proof of the following result was inspired by \cite{Wi89}.

%%%%%%
\begin{lemma}[The Dynkin-Specht-Wever Lemma]
\label{lem:DSW}
Let $d$ be a derivation of $\field\{X\}$ that preserves  $\Prim(\field\{X\})$, $u \in \field\{X\}$ and  $a\in \Prim(\field\{X\})$. We have
\begin{displaymath}
	\gamma_d(ua) = \epsilon(u)d(a) + \sum \langle u_{(1)};a, \gamma_d(u_{(2)})\rangle.
\end{displaymath}
\end{lemma}
%%%%%%
%%%%%%
\begin{proof}
Let us compute $d(ua)$ in two ways:
\begin{displaymath}
	d(ua) = \left\{ \begin{array}{l}\sum u_{(1)}\gamma_d(u_{(2)}a) + \sum (u_{(1)}a)\gamma_d(u_{(2)})\\ \\ d(u)a + ud(a) = \sum (u_{(1)} \gamma_d(u_{(2)}))a + ud(a)\end{array}\right.
\end{displaymath}
so that by (\ref{eq:brackets}) $$\sum u_{(1)}\gamma_d(u_{(2)}a) = \sum u_{(1)} \langle u_{(2)}; a, \gamma_d(u_{(3)})\rangle + \sum u_{(1)} \epsilon(u_{(2)}) d(a).$$
Using (\ref{eq:left_division}), divide by $u\1$ to get the result.
\end{proof}
%%%%%%

%%%%%%
\begin{example}
Let us compute the expansion $\log_l(\exp_l(x)\exp_l(y))$ up to degree $3$ in terms of the Shestakov-Umirbaev operations with the help of the Dynkin-Specht-Wever Lemma. Since, up to the summands of degree $\geq 5$, we have
\begin{align*}
	\log_{l}(1+x) &= \\
	& \hskip -1cm  x- \frac{1}{2}x^2 + \frac{1}{12} x^2 x + \frac{1}{4}xx^2 - \frac{1}{24}x(x^2x) -\frac{1}{8}x(xx^2) -\frac{1}{24}x^2 x^2 - \frac{1}{24}(xx^2)x + \cdots,
\end{align*}
the expansion of $\log_{l}(\exp_l(x)\exp_l(y))$ up to degree $3$ is
\begin{multline} \label{eq:BCH3}
	x+y + \frac{1}{2}[x,y]
	+ \frac{1}{3}x^2y  - \frac{1}{4}x(xy)+\frac{1}{4} x(yx) -\frac{5}{12} (xy)x + \frac{1}{12}(yx)x\\
	+ \frac{1}{2} xy^2 -\frac{5}{12}(xy)y +\frac{1}{12}(yx)y - \frac{1}{4}y(xy) - \frac{1}{6}y^2x +\frac{1}{4}y(yx)  + \cdots
\end{multline}
Now, apply Lemma \ref{lem:DSW} with $d(u) := \vert u\vert u$, where $\vert u \vert$ denotes the degree of $u$, for homogeneous $u \in \field\{x,y\}$. 
First, observe that $\gamma_d(ab) = \langle b,a\rangle$, $$\gamma_d((ab)c) = \langle c, \langle b,a\rangle\rangle + \langle a; c,b\rangle + \langle b; c,a\rangle$$ and  $$\gamma_d(a(bc)) = \gamma_d((ab)c - (a,b,c)) = \langle c, \langle b,a\rangle\rangle + \langle a; c,b\rangle + \langle b; c,a\rangle - 3(a,b,c).$$
Applying $\gamma_d$ to the homogeneous summands in (\ref{eq:BCH3}) and dividing by their degree, we can write (\ref{eq:BCH3}) as
\begin{displaymath}
	x+y + \frac{1}{2}[x,y]
	+\frac{1}{12}[x,[x,y]] - \frac{1}{3}\langle x;x,y\rangle 
	-\frac{1}{12} [y,[x,y]] + \frac{1}{6} \langle y; y,x \rangle -\frac{1}{2}\Phi(x;y,y) +\cdots
\end{displaymath}
\qed
\end{example}
%%%%%%
%%%%%%%%%%%%%%%%%%%%%%%%%%%%%%%%%%%%%%%%%%%%%%%%%%%%%%%%%%%%%%%%%%%%%%%%%%%%%%%%%%%%%%%%
%    __  ___                                                                 _           
%   /  |/  /___ _____ _____  __  _______   ___  _  ______  ____ _____  _____(_)___  ____ 
%  / /|_/ / __ `/ __ `/ __ \/ / / / ___/  / _ \| |/_/ __ \/ __ `/ __ \/ ___/ / __ \/ __ \
% / /  / / /_/ / /_/ / / / / /_/ (__  )  /  __/>  </ /_/ / /_/ / / / (__  ) / /_/ / / / /
%/_/  /_/\__,_/\__, /_/ /_/\__,_/____/   \___/_/|_/ .___/\__,_/_/ /_/____/_/\____/_/ /_/ 
%             /____/                             /_/                                     
\subsection{A non-associative Magnus expansion.} 
%%%%%%%%%%%%%%%%%%%%%%%%%%%%%%%%%%%%%%%%%%%%%%%%%%%%%%%%%%%%%%%%%%%%%%%%%%%%%%%%%%%%%%%%
The differential equation 
\begin{displaymath}
	X'(t) = A(t)X(t)
\end{displaymath} 
when $X(t)$ and $A(t)$ do not necessarily commute (for instance, $X(t)$ may belong to a matrix Lie group and $A(t)$ to the corresponding Lie algebra) has been studied since long ago \cite{BlaCaOtRo09}. 
A fruitful approach is to look for solutions of the form $X(t) = \exp(\Omega(t))$ for some $\Omega(t)$, where $\exp(x)$ denotes the usual exponential. 
The solution $\Omega(t)$ is determined by the initial condition and by the differential equation 
\begin{equation}\label{eq:Magnus_recursion}
	\Omega'(t) = \frac{\ad_{\Omega(t)}}{\exp(\ad_{\Omega(t)})-\Id}(A(t)) = \sum_{n=0}^{\infty} \frac{B_n}{n!}\ad_{\Omega(t)}^n(A(t)),
\end{equation}
where $B_n$ denotes the $n$-th Bernoulli number. Take $X(t) := \exp(tx)\exp(y)$; then  
\begin{displaymath}
	 X'(t) = (x\exp(tx))\exp(y) \stackrel{<1>}{=} x(\exp(tx)\exp(y)) =  xX(t)
\end{displaymath}
so we can  use (\ref{eq:Magnus_recursion}) in order to study $\Omega(t) = \log(\exp(tx)\exp(y))$.
However, in a non-associative setting there are some details to be taken care of, since, for instance, equality $<$1$>$ above requires the associativity. 
%%%%%%
\begin{proposition}
Let $H$ be a unital algebra, $e(x) \in \field\{\{x\}\}$ a base for logarithms and $X(t) := e(\Omega(t))$ with $\Omega(t) \in H[[t]]$ such that $\Omega(0) = 0$. For any $A(t) \in H[[t]]$ the solution $\Omega(t)$ to the equation
\begin{displaymath}
	X'(t) = X(t)A(t) 
\end{displaymath}
satisfies
\begin{displaymath}
	\Omega'(t)=(\tau^{e(\Omega(t))})^{-1}(A(t))
\end{displaymath}
where $\tau^{e(x)}$ is defined by
\begin{displaymath}
	\tau^{e(x)}(y) := e(x) \left\backslash \left.\frac{d}{ds}\right\vert_{s=0} e(x+sy) \right.\in   \field\{\{x,y\}\}.
\end{displaymath}
\end{proposition}
%%%%%%
%%%%%%
\begin{proof}
Evaluating at $x = \Omega(t)$ and $y = \Omega'(t)$ we get 
\begin{align*}
	\tau^{e(\Omega(t))} (\Omega'(t))  &= e(\Omega(t))  \left\backslash \left.\frac{d}{ds}\right\vert_{s=0} e(\Omega(t)+s \Omega'(t))\right. =  e(\Omega(t))\left\backslash \frac{d}{dt}e(\Omega(t))\right. \\ &= X(t)\backslash X'(t) = A(t).
\end{align*}
If $x = 0$ then $\tau^{e(x)}(y) =  \left.\frac{d}{ds}\right\vert_{s=0} e(sy) = \alpha y$ for some $0 \neq \alpha \in \field$ and there exists $(\tau^{e(x)})^{-1}(y) \in \field\{\{x,y\}\}$ such that $(\tau^{e(x)})^{-1}(\tau^{e(x)}(y)) = y$. Therefore $\Omega'(t)=(\tau^{e(\Omega(t))})^{-1}(A(t))$.
\end{proof}
%%%%%%

In order to compute $(\tau^{e(x)})^{-1}(y)$ in terms of the Shestakov-Umirbaev operations, we will use the Dynkin-Specht-Wever Lemma.  
Consider the derivation $y\partial_x$ of $\field\{\{x,y\}\}$ determined by 
\begin{equation}\label{eq:derivation ydeltax}
	(y\partial_x)(x) := y \quad \text{and} \quad (y\partial_x)(y) := 0.
\end{equation}
By induction on the degree $\vert u\vert$ of $u$ we can check that 
$$\Delta((y\partial_x)(u)) = \sum (y\partial_x)(u\1) \otimes  u\2 + u\1 \otimes (y\partial_x)(u\2)$$ so that $(y\partial_x)$  preserves $\Prim(\field\{\{x,y\}\})$  and it is related to $\tau^{e(x)}(y)$ via
\begin{displaymath}
	\tau^{e(x)}(y)=  e(x) \left\backslash \left.\frac{d}{ds}\right\vert_{s=0} e(x+sy)\right. = e(x) \backslash (y\partial_x)(e(x))= \gamma_{y\partial_x}(e(x)).
\end{displaymath}
Now, in order to apply the Dynkin-Specht-Wever Lemma recursively $e(x)$ should be a linear combination of left-normed products of primitive elements. 
This is the main reason for restricting ourselves to $\exp_l(x)$.
%%%%%%
\begin{lemma}\label{lem:tau}
The component $\tau_n$ of degree $n$ in $x$ of $\tau^{\exp_l(x)}(y)$ is 
\begin{displaymath}
	\sum_{i=1}^n \frac{1}{n+1}\frac{1}{(n-i)!} \langle x^{n-i}; x, \tau_{i-1} \rangle
\end{displaymath}
where $\tau_0:= y$.
\end{lemma}
%%%%%%

The expansion of $(\tau^{\exp_l(x)})^{-1}(y)$ can be easily obtained from the expansion of $\tau^{\exp_l(x)}(y)$. 
Given a tuple $J = (j_1,\dots, j_s) \in \Z^s$ with $j_1,\dots, j_s \geq 1$ define
\begin{align*}
	P_J(x;y) &:=\langle \underbrace{x,\dots,x}_{j_1 -1};x,\langle  \underbrace{x,\dots,x}_{j_2 -1};x,\langle \dots\langle  \underbrace{x,\dots,x}_{j_s -1};x,y\rangle\rangle\rangle \quad \text{and}\\
	m_J &:= \frac{1}{j_1+\cdots + j_s+1}\frac{1}{(j_1-1)!}\frac{1}{j_2+\cdots+j_s+1}\frac{1}{(j_2-1)!}\cdots \frac{1}{j_s+1}\frac{1}{(j_s-1)!}.\\
\end{align*}
The concatenation $(i_1,\dots, i_r,j_1,\dots,j_s)$ of $(i_1,\dots, i_r)$ and $(j_1,\dots, j_s)$ will be denoted by $(i_1,\dots,i_r)\vert\vert (j_1,\dots, j_s)$. 

%%%%%%
\begin{theorem}\label{thm:degOneExpansion}
In $\field\{\{x,y\}\}$ we have
\begin{displaymath}
	(\tau^{\exp_l(x)})^{-1}(y) = y + \sum_{J} n_J P_J(x;y)
\end{displaymath}
where $J$ runs over all possible tuples with entries $\geq 1$ and 
\begin{displaymath}
	n_J := \sum_{J=J_1 \vert\vert \cdots \vert\vert J_l} (-1)^l m_{J_1}\cdots m_{J_l}.
\end{displaymath}
\end{theorem}
%%%%%%
%%%%%%
\begin{proof}
Let $y\partial_x$ be the derivation of $\field\{\{x,y\}\}$ determined by (\ref{eq:derivation ydeltax}). The Dynkin-Specht-Wever Lemma implies
\begin{multline}\nonumber
	\frac{1}{(n+1)!} \gamma_{y\partial_x}(x^{n+1}) = \frac{1}{(n+1)!} \sum_{\substack{J = (j_1,\dots, j_s)\\ j_1+\cdots+j_s = n}} {\binom{n}{j_1 -1}{\binom{n-j_1}{j_2 -1}}}\cdots \\
	\cdots {\binom{n-j_1-\cdots -j_s}{j_s-1}} P_J(x;y)
\end{multline}
so $\tau^{\exp_l(x)}(y) = y + \sum_J m_J P_J(x;y)$ and
\begin{displaymath}
	(\tau^{\exp_l(x)})^{-1}(y) = y + \sum_{\substack{l\geq 1 \\J_1,\dots,J_l}} (-1)^l m_{J_1}\cdots m_{J_l}P_{J_1}(x;P_{J_2}(x;\cdots (P_{J_l}(x;y)))).
\end{displaymath}
Since $P_{J_1}(x;P_{J_2}(x;\cdots (P_{J_l}(x;y)))) = P_{J_1\vert\vert \cdots \vert\vert J_l}(x;y)$, the result follows.
\end{proof}
%%%%%%

%%%%%%\
\begin{corollary}\label{cor:Magnus}
Let $H$ be a unital algebra and $A(t) \in H[[t]]$. The solution $\Omega(t) \in H[[t]]$ of the equation 
\begin{displaymath}
	X'(t) = X(t)A(t)
\end{displaymath}
with $X(t) := \exp_l(\Omega(t))$ and $\Omega(0) = 0$ satisfies
\begin{equation}\label{eq:nonassociative_Magnus_recursion}
	\Omega'(t)= A(t) + \sum_{J} n_J P_J(\Omega(t);A(t))
\end{equation}	
where $J$ runs over all possible tuples with the components $\geq 1$. 
\end{corollary}

%%%%%%%%%%%%%%%%%%%%%%%%%%%%%%%%%%%%%%%%%%%%%%%%%%%%%%%%%%%%%%%%%%%%
%    ____  ________  __   ______                           __     
%   / __ )/ ____/ / / /  / ____/___  _________ ___  __  __/ /___ _
%  / __  / /   / /_/ /  / /_  / __ \/ ___/ __ `__ \/ / / / / __ `/
% / /_/ / /___/ __  /  / __/ / /_/ / /  / / / / / / /_/ / / /_/ / 
%/_____/\____/_/ /_/  /_/    \____/_/  /_/ /_/ /_/\__,_/_/\__,_/  
%                                                                 
\subsection{A non-associative Baker-Campbell-Hausdorff formula.}
%%%%%%%%%%%%%%%%%%%%%%%%%%%%%%%%%%%%%%%%%%%%%%%%%%%%%%%%%%%%%%%%%%%%
We will use the formula for $(\tau^{\exp_l(x)})^{-1}$ in Theorem~\ref{thm:degOneExpansion} to describe, in terms of the Shestakov-Umirbaev operations, the differential equation satisfied by $\log_l(\exp_l(x)\exp_l(ty))$.
%%%%%%
\begin{proposition}\label{prop:BCH} 
Let $e(x)$ be a normalized base for logarithms. In $\field\{\{x,y\}\}$ we have
\begin{displaymath}
	\log_{e}(e(x)e(y)) = x + (\tau^{e(x)})^{-1} (y) + O(y^2).
\end{displaymath}
\end{proposition}
%%%%%%
\begin{proof}
Consider $\Omega(t) := \log_{e}(e(x)e(ty)) = x + \Omega'(0) t + O(t^2)$. Since
\begin{align*}
	\tau^{e(\Omega(t))}(\Omega'(t)) &= e(\Omega(t))\left\backslash \frac{d}{dt}e(\Omega(t))\right. = (e(x)e(ty)) \left\backslash \frac{d}{dt}(e(x)e(ty))\right. \\
	&= (e(x) e(ty)) \left\backslash \left(e(x) \frac{d}{dt}e(ty)\right),\right.
\end{align*}
evaluating at $t = 0$, we get $\tau^{e(x)}(\Omega'(0)) = y$ so $\Omega'(0)=(\tau^{e(x)})^{-1} (y)$.
\end{proof}
%%%%%%
In the case when $e(x)$ is $\exp_l(x)$ or $\exp_r(x)$, Proposition~\ref{prop:BCH} was proved in \cite{GW}.
%%%%%%
\begin{example}
The components of degree $0, 1, 2 $ and $3$ of $\tau^{\exp_l(x)}(y)$ are $\tau_0 = y$,  $\tau_1 = \frac{1}{2}\langle x,y \rangle$, $\tau_2 = \frac{1}{3}\langle x; x,y \rangle + \frac{1}{6}\langle x, \langle x,y\rangle \rangle$ and $\tau_3 = \frac{1}{8}\langle x,x; x,y \rangle + \frac{1}{8}\langle x;x,\langle x,y\rangle\rangle + \frac{1}{12}\langle x,\langle x;x,y\rangle \rangle + \frac{1}{24}\langle x,\langle x,\langle x,y\rangle\rangle\rangle$.
Thus, the component of degree one in $y$ in $\log_l(\exp_l(x)\exp_l(y))$ is 
\begin{multline}\nonumber
	 y- \frac{1}{2}\langle x,y\rangle + \left(\frac{1}{12}\langle x,\langle x,y\rangle \rangle - \frac{1}{3}\langle x;x,y\rangle\right) 
	 +\\ \left(\frac{1}{12}\langle x,\langle x;x,y\rangle\rangle + \frac{1}{24}\langle x;x,\langle x,y\rangle\rangle - \frac{1}{8}\langle x,x;x,y\rangle \right)  + \cdots 
\end{multline}

We can compute directly the coefficient of $\langle x;x,\langle x,y\rangle \rangle$ in $\log_l(\exp_l(x)\exp_l(y))$, for instance. 
Since $\langle x;x,\langle x,y\rangle \rangle = P_{(2,1)}(x;y)$ then Theorem~\ref{thm:degOneExpansion} ensures that this coefficient equals $n_{(2,1)} = m_{(2)}m_{(1)} - m_{(2,1)} = \frac{1}{3}\frac{1}{1!}\frac{1}{2}\frac{1}{0!} - \frac{1}{4}\frac{1}{1!}\frac{1}{2}\frac{1}{0!}=\frac{1}{24}$.
\ \hfill $\qed$
\end{example}
%%%%%%
%%%%%%
\begin{proposition}[Magnus expansion for the Baker-Campbell-Hausdorff formula]\label{prop:BCHrecurr}
Let $\Omega(t) := \log_l(\exp_l(x)\exp_l(ty))$. In $\field\{\{x,y\}\}[[t]]$ we have
\begin{multline}\nonumber
	\Omega'(t) = \left(y+ \sum_Jn_JP_J(\Omega(t);y)\right)\\
	\quad- \left( \Phi(\exp_l(x);\exp_l(ty);y) + \sum_JP_J(\Omega(t);\Phi(\exp_l(x);\exp_l(ty);y)\right).
\end{multline}
\end{proposition}
%%%%%%
%%%%%%
\begin{proof}
We have
\begin{align*}
	\tau^{\exp_l(\Omega(t))}(\Omega'(t)) &= \exp_l(\Omega(t))\left\backslash \frac{d}{dt}\exp_l(\Omega(t))\right. \\
	&= (\exp_l(x) \exp_l(ty)) \left\backslash \left(\exp_l(x) \frac{d}{dt}\exp_l(ty)\right)\right.\\
	&=  (\exp_l(x) \exp_l(ty)) \left\backslash (\exp_l(x) (\exp_l(ty)y))\right.\\
	&=  y - p(\exp_l(x); \exp_l(ty);y) \\
	&= y - \Phi(\exp_l(x);\exp_l(ty);y)
\end{align*}
so $\Omega'(t) = (\tau^{\exp_l(\Omega(t))})^{-1}(y) - (\tau^{\exp_l(\Omega(t))})^{-1}(\Phi(\exp_l(x);\exp_l(ty);y))$. 
The result follows from Theorem~\ref{thm:degOneExpansion}.
\end{proof}
%%%%%%
%%%%%%
\begin{example}
The component of $\log_l(\exp_l(x)\exp_l(y))$ of degree $1$ in $x$ and degree $2$ in $y$ is 
\begin{align*}
	\Omega_{1,2} &:= \frac{1}{2}\left( n_{(1)}\langle \Omega_{1,1},y\rangle  + n_{(1,1)}\langle y,\langle x,y\rangle\rangle + n_{(2)}\langle y;x,y\rangle  -\Phi(x;y,y)\right)\\
	& = -\frac{1}{12}[y,[x,y]] +  \frac{1}{6}\langle y;y,x\rangle - \frac{1}{2}\Phi(x;y,y).
\end{align*}
\ \hfill\qed
\end{example}
%%%%%%

Based on Proposition~\ref{prop:BCHrecurr} we can compute the initial terms of the expansion of  $\log_l(\exp_l(x)\exp_l(y))$. 

%%%%%%
\begin{theorem}[Non-associative Baker-Campbell-Hausdorff Formula]\label{thm:BCH}
The expansion of $\log_l(\exp_l(x)\exp_l(y))$ in $\field\{\{x,y\}\}$ is
%
%\begin{align*}
%		& x+y + \frac{1}{2}[x,y]\\
%		& \quad +\frac{1}{12}[x,[x,y]] - \frac{1}{3}\langle x;x,y\rangle 
%		-\frac{1}{12} [y,[x,y]] - \frac{1}{6} \langle y; x,y \rangle -\frac{1}{2}\Phi(x;y,y) \\
%		&\quad -\frac{1}{24}\langle x; x, [x,y]\rangle - \frac{1}{12} [x,\langle x;x,y\rangle] - \frac{1}{8}\langle x,x;x,y\rangle \\
%		&\quad +\frac{1}{48}[[x,[x,y]],y] + \frac{1}{48}[x,[[x,y],y]]-\frac{1}{48}\langle y;x,[x,y]\rangle - \frac{1}{24}[x,\langle y;x,y\rangle]\\
%		&\quad -\frac{1}{4}\Phi(x,x;y,y) - \frac{1}{4}[x,\Phi(x;y,y)]-\frac{1}{24}[\langle x;x,y\rangle,y]\\
%		&\quad -\frac{1}{16}\langle x; [x,y],y\rangle - \frac{7}{48}\langle x,y;x,y\rangle + \frac{1}{48}\langle y,x;x,y\rangle \\
%		& \quad + \frac{1}{12} [\Phi(x;y,y),y] + \frac{1}{24}\langle y; y,[x,y]\rangle - \frac{1}{24} \langle y,y;x,y\rangle - \frac{1}{6}\Phi(x;y,y,y)
%\end{align*}
\begin{align*}
		& x+y + \frac{1}{2}[x,y]\\
		& \quad +\frac{1}{12}[x,[x,y]] - \frac{1}{3}\langle x;x,y\rangle 
		-\frac{1}{12} [y,[x,y]] - \frac{1}{6} \langle y; x,y \rangle -\frac{1}{2}\Phi(x;y,y) \\
		&\quad -\frac{1}{24}\langle x; x, [x,y]\rangle - \frac{1}{12} [x,\langle x;x,y\rangle] - \frac{1}{8}\langle x,x;x,y\rangle \\
		&\quad +\frac{1}{24}[[x,[x,y]],y] - \frac{1}{24}[x,\langle y;x,y\rangle] -\frac{1}{4}\Phi(x,x;y,y) - \frac{1}{4}[x,\Phi(x;y,y)] \\
		&\quad -\frac{1}{24}[\langle x;x,y\rangle,y] -\frac{1}{24}\langle x; [x,y],y\rangle - \frac{1}{6}\langle x,y;x,y\rangle + \frac{1}{24}\langle y,x;x,y\rangle \\
		& \quad + \frac{1}{12} [\Phi(x;y,y),y] + \frac{1}{24}\langle y; y,[x,y]\rangle - \frac{1}{24} \langle y,y;x,y\rangle - \frac{1}{6}\Phi(x;y,y,y)
\end{align*}
plus terms of degree $\geq 5$.
\end{theorem}
%$[x,[[x,y],y]] = [[x,[x,y]],y] + \langle y; x, [x,y]\rangle + \langle x;[x,y],y\rangle - \langle x,y; x,y\rangle$
%%%%%%
%%%%%%%%%%%%%%%%%%%%%%%%%%%%%%%%%%%%%%%%%%%%%%%%%%%%%%%%%%%%%%%%%%%%%%%%%%%%%%%%%%%%%%%
%   ___                             _ _ _                         _                   
%  / __\ ___ _ __ _ __   ___  _   _| | (_)  _ __  _   _ _ __ ___ | |__   ___ _ __ ___ 
% /__\/// _ \ '__| '_ \ / _ \| | | | | | | | '_ \| | | | '_ ` _ \| '_ \ / _ \ '__/ __|
%/ \/  \  __/ |  | | | | (_) | |_| | | | | | | | | |_| | | | | | | |_) |  __/ |  \__ \
%\_____/\___|_|  |_| |_|\___/ \__,_|_|_|_| |_| |_|\__,_|_| |_| |_|_.__/ \___|_|  |___/
%                                                                                      
\section{A connection with Bernoulli numbers and binary trees}
%%%%%%%%%%%%%%%%%%%%%%%%%%%%%%%%%%%%%%%%%%%%%%%%%%%%%%%%%%%%%%%%%%%%%%%%%%%%%%%%%%%%%%%
Formulas (\ref{eq:Magnus_recursion}) and (\ref{eq:nonassociative_Magnus_recursion}) give an alternative point of view on the relation between the numbers $\{n_J\}_J$ and $\{B_k\}_k$ which has been established in \cites{Wo97,Fu00}. 

%%%%%%
\begin{theorem}
We have
\begin{displaymath}
	\frac{B_k}{k!} = n_{\underbrace{(1,\dots,1)}_{k}}.
\end{displaymath}
\end{theorem}
%%%%%%
%%%%%%
\begin{proof}
By definition, $\langle x_1,\dots,x_m; y,z \rangle$ and $\Phi(x_1,\dots, x_m; y_1,\dots, y_n; y_{n+1})$ vanish in any associative algebra, with the only exception of $\langle y,z\rangle $. Thus, after projecting from $\field\{\{x,y\}\}$, in  $\field \langle\langle x,y \rangle\rangle$ we get 
\begin{align*}
(\tau^{\exp(x)})^{-1}(y) &= y + \sum_{J = (1,\dots, 1)} n_J P_J(x;y) = \sum_{k=0}^\infty n_{\underbrace{(1,\dots,1)}_{k}} (-1)^k \ad_{x}^k(y).
\end{align*}
Since it is well-known that $(\tau^{\exp(x)})^{-1}(y) = \sum_{k=0}^{\infty} (-1)^k\frac{B_k}{k!}\ad_{x}^k(y)$ holds in $\field \langle\langle x,y \rangle\rangle$  we get the result; the sign $(-1)^k$ in the latter formula comes from our choice $\tau^{\exp(x)} := \exp(x) \backslash \left.\frac{d}{ds}\right\vert_{s=0} \exp(x+s y)$ instead of $\left.\frac{d}{ds}\right\vert_{s=0} \exp(x+s y)  / \exp(x)$.
\end{proof}
%%%%%%
In \cite{Wo97} Woon gave an algorithm to compute $B_n/n!$ with the help of the binary tree 
\begin{center}
	\begin{tikzpicture}
		\node (v1) at (-1,3) {$[1,2]$};
		\node (v2) at (-3,2) {$[-1,3]$};
		\node (v3) at (1,2) {$[1,2,2]$};
		\node at (5,3) {level 1};
		\node at (5,2) {level 2};
		\node (v4) at (-4,1) {$[1,4]$};
		\node (v5) at (-2,1) {$[-1,2,3]$};
		\node (v6) at (0,1) {$[-1,3,2]$};
		\node (v7) at (2,1) {$[1,2,2,2]$};
		\draw  (v1) edge (v2);
		\draw  (v1) edge (v3);
		\draw  (v2) edge (v4);
		\draw  (v2) edge (v5);
		\draw  (v3) edge (v6);
		\draw  (v3) edge (v7);
		\node at (-1,1) {$\vdots$};
		\node at (5,1) {$\vdots$};
	\end{tikzpicture}
\end{center}
Here, the nodes are labeled by $[a_1,\dots, a_r]$; 
the root is $[1,2]$  and at any node  we have
\begin{center}
	\begin{tikzpicture}
		\node (v1) at (0,2) {$[a_1,\dots, a_r]$};
		\node (v2) at (-2,1) {$[-a_1,a_2+1,\dots, a_r]$};
		\node (v3) at (2,1) {$[a_1,2,a_2,\dots, a_r]$};
		\draw  (v1) edge (v2);
		\draw  (v1) edge (v3);
	\end{tikzpicture}
\end{center}
The factorial of the node $N=[a_1,\dots, a_r]$ is $N! := a_1 (a_2!\cdots a_r!)$. 
Woon proved the equality
\begin{displaymath}
	\frac{B_k}{k!} = \sum_{N} \frac{1}{N!}
\end{displaymath}
for $k \geq 2$, where $N$ runs over the nodes in the level $k$. In \cite{Fu00} Fuchs extended this construction as follows. Consider the \emph{general PI binary tree}
\begin{center}
	\begin{tikzpicture}
		\node (v1) at (-1,3) {$(1)$};
		\node (v2) at (-3,2) {$(1,1)$};
		\node (v3) at (1,2) {$(2)$};
		\node at (5,3) {level 1};
		\node at (5,2) {level 2};
		\node at (5,1) {$\vdots$};
		\node (v4) at (-4,1) {$(1,1,1)$};
		\node (v5) at (-2,1) {$(2,1)$};
		\node (v6) at (0,1) {$(1,2)$};
		\node (v7) at (2,1) {$(3)$};
		
		\draw  (v1) edge (v2);
		\draw  (v1) edge (v3);
		\draw  (v2) edge (v4);
		\draw  (v2) edge (v5);
		\draw  (v3) edge (v6);
		\draw  (v3) edge (v7);
		\node at (-1,1) {$\vdots$};
	\end{tikzpicture}
\end{center}
with root $(1)$ and at each node
\begin{displaymath}
	\begin{tikzpicture}
		\node (v1) at (0,2) {$(a_1,\dots, a_r)$};
		\node (v2) at (-2,1) {$(1,a_1,a_2,\dots, a_r)$};
		\node (v3) at (2,1) {$(a_1+1,a_2,\dots, a_r)$};
		\draw  (v1) edge (v2);
		\draw  (v1) edge (v3);
	\end{tikzpicture}
\end{displaymath}
For any sequence $(c_n)_{n\geq 1}$ of complex numbers change the node $(a_1,\dots, a_r)$ by $c_{a_1}\cdots c_{a_r}$. 
Then define $x_k$ to be the sum of the nodes in the $k$-th level. 
This value depends on the sequence $(c_n)_{n\geq 1}$. 
In the case when $c_n = \frac{-1}{n+1!}$  we get the tree
\begin{center}
	\begin{tikzpicture}
		\node (v1) at (-1,3) {$-\frac{1}{2!}$};
		\node (v2) at (-3,2) {$\frac{1}{2!2!}$};
		\node (v3) at (1,2) {$-\frac{1}{3!}$};
		\node (v4) at (-4,1) {$-\frac{1}{2!2!2!}$};
		\node (v5) at (-2,1) {$\frac{1}{3!2!}$};
		\node (v6) at (0,1) {$\frac{1}{2!3!}$};
		\node (v7) at (2,1) {$-\frac{1}{4!}$};
		\draw  (v1) edge (v2);
		\draw  (v1) edge (v3);
		\draw  (v2) edge (v4);
		\draw  (v2) edge (v5);
		\draw  (v3) edge (v6);
		\draw  (v3) edge (v7);
		\node at (-1,1) {$\vdots$};
	\end{tikzpicture}
\end{center}
and for each $k$ we have  $x_k = B_k / k!$.  

To relate these constructions to the numbers $n_J$ in Theorem~\ref{thm:degOneExpansion} we use a binary tree to collect the summands involved in $n_J = \sum_{J=J_1 \vert\vert \cdots \vert\vert J_l} (-1)^l m_{J_1}\cdots m_{J_l}$. 
Consider associative but non-commutative indeterminates $x_1,x_2,\dots$ and the tree
\begin{center}
	\begin{tikzpicture}
		\node (v1) at (-1,3) {$(x_1)$};
		\node (v2) at (-3,2) {$(x_1,x_2)$};
		\node (v3) at (1,2) {$(x_1x_2)$};
		\node (v4) at (-4,1) {$(x_1,x_2,x_3)$};
		\node (v5) at (-2,1) {$(x_1,x_2x_3)$};
		\node (v6) at (0,1) {$(x_1x_2,x_3)$};
		\node (v7) at (2,1) {$(x_1x_2x_3)$};
		\draw  (v1) edge (v2);
		\draw  (v1) edge (v3);
		\draw  (v2) edge (v4);
		\draw  (v2) edge (v5);
		\draw  (v3) edge (v6);
		\draw  (v3) edge (v7);
		\node at (-1,1) {$\vdots$};
		\node at (5,3) {level 1};
		\node at (5,2) {level 2};
		\node at (5,1) {$\vdots$};
	\end{tikzpicture}
\end{center}
where at any node on the level $n-1$ we have
\begin{center}
	\begin{tikzpicture}
		\node (v1) at (0,2) {$(w_1,\dots, w_r)$};
		\node (v2) at (-2,1) {$(w_1,\dots, w_r,x_n)$};
		\node (v3) at (2,1) {$(w_1,\dots, w_rx_n)$};
		\draw  (v1) edge (v2);
		\draw  (v1) edge (v3);
	\end{tikzpicture}
\end{center}
Consider a sequence of numbers $a_1,a_2,\dots$. 
Define for $w=x_{i_1}\cdots x_{i_s}$ the number $m_w = m_{x_{i_1}\cdots x_{i_s}} = - m_{(a_{i_1},\dots, a_{i_s})}$ and replace any node $(w_1,\dots, w_r)$ with $m_{(w_1,\dots, w_r)} = m_{w_1}\cdots m_{w_r}$. 
The sum of the nodes in the level $n$ of the resulting tree is $n_{(a_1,\dots, a_n)}$. 
In case that $ a_1 = a_2 = \cdots =1$, in the previous construction we can replace the label $x_{i_1}\cdots x_{i_s}$ by $s$  without losing information. 
With these new labels, at any node on the level $n-1$ of the tree we have
\begin{center}
	\begin{tikzpicture}
		\node (v1) at (0,2) {$(a_1,\dots, a_r)$};
		\node (v2) at (-2,1) {$(a_1,\dots, a_r,1)$};
		\node (v3) at (2,1) {$(a_1,\dots, a_r+1)$};
		\draw  (v1) edge (v2);
		\draw  (v1) edge (v3);
	\end{tikzpicture}
\end{center}
which essentially gives the general PI binary tree. 
The number that we attach to the node $(a_1,\dots,a_r)$ is $(-m_{\underbrace{(1,\dots,1)}_{a_1}})\cdots (-m_{\underbrace{(1,\dots,1)}_{a_r}}) = \frac{-1}{(a_1+1)!}\cdots \frac{-1}{(a_r+1)!}$, so we recover the construction of Fuchs.
%%%%%%%%%%%%%%%%%%%%%%%%%%%%%%%%%%%%%%%%%%%%%%%%%%%%%%%%%%%%%%
%   ___    ___                     _   _   _                 
%  / __\  / __\ /\  /\   ___ _   _| |_| |_(_)_ __   __ _ ___ 
% /__\// / /   / /_/ /  / __| | | | __| __| | '_ \ / _` / __|
%/ \/  \/ /___/ __  /  | (__| |_| | |_| |_| | | | | (_| \__ \
%\_____/\____/\/ /_/    \___|\__,_|\__|\__|_|_| |_|\__, |___/
%                                                  |___/     
\section{BCH-cuts}
%%%%%%%%%%%%%%%%%%%%%%%%%%%%%%%%%%%%%%%%%%%%%%%%%%%%%%%%%%%%%%
In this section we briefly study the coefficients of the non-associative monomials in the series $\log_l(\exp_l(x)\exp_l(y))$. 
Unfortunately, these monomials are not left-normed so we cannot directly apply the Dynkin-Specht-Wever Lemma to them to get a closed form of the Baker-Campbell-Hausdorff formula similar to (\ref{eq:Dynkin_formula}).

Given a monomial $w(x) \in \field\{x\}$, that we can identify with a binary planar rooted tree,  any tuple $c=(\tau(x),\tau_1(x),\dots, \tau_l(x))$ of monomials satisfying $w(x) = \tau(\tau_1(x),\dots, \tau_l(x))$ with $\vert \tau_i(x)\vert \geq 1$ will be called a \emph{cut} of $w(x)$ -- recall that $\vert \tau(x) \vert$ denotes the degree of $\tau(x)$ in $x$. 
Attached to $\tau_i(x)$ there is the set $\lambda(\tau_i(x)) := \{\vert\tau_1(x)\vert + \cdots + \vert \tau_{i-1}(x) \vert + 1,\dots, \vert \tau_1(x)\vert + \cdots + \vert \tau_i(x)\vert \}$. 
The pair $(\tau_i,\lambda(\tau_i))$ is a \emph{branch} of $w(x)$. 
Sometimes we will refer to $\tau_i$ as a branch of $w(x)$, although this is an abuse of notation since we should specify the positions that the branch occupies inside $w(x)$. 
Monomials $w(x,y) \in \field\{x,y\}$ are represented by binary planar rooted trees with leaves decorated with $x$ or $y$.  

%%%%%%
\begin{lemma}
\label{lem:branches}
Let $\tau'(x),\tau''(x)$ be branches of $w(x)$ such that $ \lambda(\tau')\cap \lambda(\tau'')\neq \emptyset$. 
Then either $\lambda(\tau')\subseteq \lambda(\tau'')$ or $\lambda(\tau'') \subseteq \lambda(\tau')$.
\end{lemma}
%%%%%%

The element
\begin{displaymath}
	\log_l(\exp_l(x)\exp_l(y))
\end{displaymath}
in $\field\{\{x,y\}\}$ is expanded in terms of monomials $\tau(x^{i_1}y^{j_1},x^{i_2}y^{j_2},\dots )$ with $i_1+j_1,i_2+j_2,\dots \geq 1$. Unfortunately, monomials $\tau(x^{i_1}y^{j_1},x^{i_2}y^{j_2},\dots )$ might represent the same monomial $w(x,y)$ for different values of $\tau$ and $i_1,j_1,\dots$ 
One way of computing the coefficient of $w(x,y)$ in the series $\log_l(\exp_l(x)\exp_l(y))$ is to determine the cuts of $w(x,y)$ where every branch is of the form $x^i y^j$ (\emph{BCH-cuts}). If we denote by $C(w)$ the set of all BCH-cuts of $w(x,y)$, the coefficient in $\log_l(\exp_l(x)\exp_l(y))$ of the monomial $w(x,y)$ is
\begin{displaymath}
	\sum_{(\tau,x^{i_1}y^{j_1},\dots,x^{i_{\vert\tau\vert}} y^{j_{\vert\tau\vert}}) \in C(w)} \frac{1}{i_1!\cdots i_{\vert \tau\vert}!}\frac{1}{j_1!\cdots j_{\vert \tau\vert}!}c_\tau.
\end{displaymath}
where $c_\tau  := \frac{B_\tau}{\tau!}$ is the coefficient of $\tau$ in $\log_l(1+x)$. The set $C(w)$ can be easily determined since Lemma \ref{lem:branches} implies that there exists a unique $(\tau,x^{i_1}y^{j_1},\dots,x^{i_{\vert\tau\vert}} y^{j_{\vert\tau\vert}})\in C(w)$ with minimal $\vert \tau \vert$. 
The branches of any $c \in C(w)$ can be obtained as the branches in a BCH-cut of the monomials $x^{i_1}y^{j_1}$, \dots, $x^{i_{\vert\tau\vert}} y^{j_{\vert\tau\vert}}$. 
Each monomial $x^iy^j$ produces $ij+1$ ($i,j\geq 1$), $i$ ($j=0$) or $j$ ($i =0$) BCH-cuts. 

%%%%%%
\begin{example}
Using $[\dots]$ to delimit branches and writing $\tau([x^{i_1}y^{j_1}],\dots,[x^{i_{\vert\tau\vert}} y^{j_{\vert\tau\vert}}])$ instead of $(\tau,x^{i_1}y^{j_1},\dots,x^{i_{\vert\tau\vert}} y^{j_{\vert\tau\vert}})$, some BCH-cuts are
\medskip
\begin{displaymath}
	\begin{array}{|c|c|}
		\hline w(x,y) & \text{BCH-cuts of } w(x,y) \\ 
		\hline\hline x^2y& [x^2y], [x^2][y], ([x][x])[y]  \\ 
		\hline  x(xy)& [x][xy], [x]([x][y]) \\ 
		\hline  x(yx)& [x]([y][x])  \\ 
		\hline (xy)x & [xy][x], ([x][y])[x] \\ 
		\hline \vdots & \vdots \\ 
		\hline 
	\end{array} 
\end{displaymath}
\begin{center}
	\includegraphics[scale=.5]{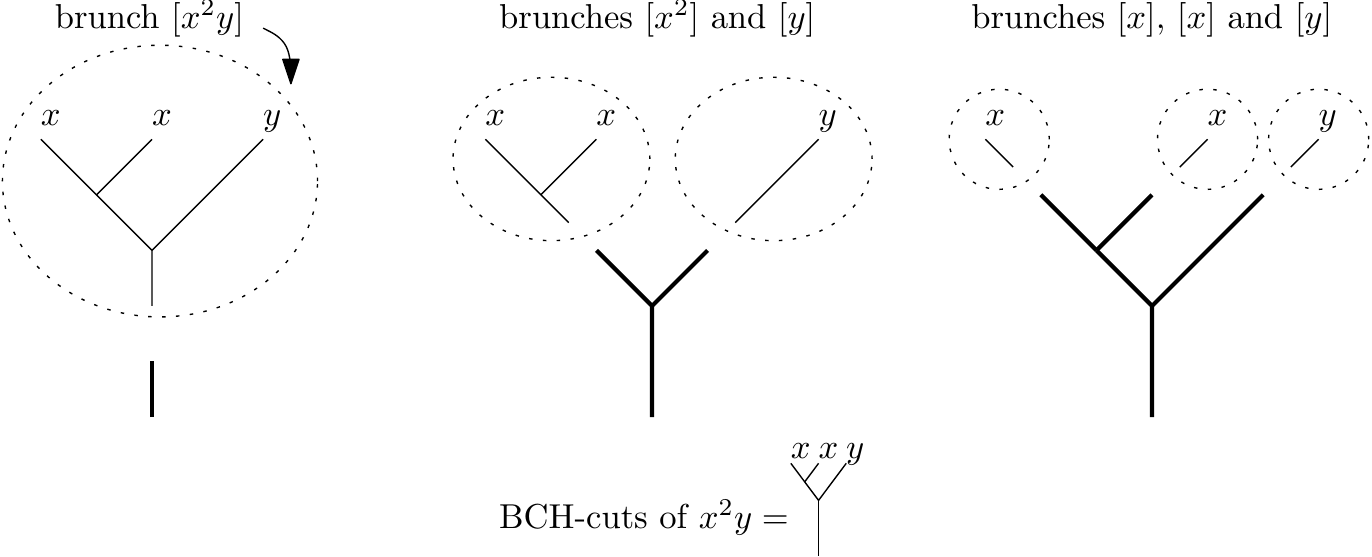}
\end{center}
Thus, for instance, the coefficient of $x(xy)$ in  $\log_l(\exp_l(x)\exp_l(y))$ is 
\begin{displaymath}
	c_{x^2} + c_{xx^2} = -\frac{1}{2}+\frac{1}{4}= -\frac{1}{4}
\end{displaymath}
while the coefficient of $x^2 y$ is
\begin{displaymath}
	\frac{c_x}{2} + \frac{c_{x^2}}{2}+ c_{x^2x} = \frac{1}{2}  -\frac{1}{4}+\frac{1}{12} = \frac{1}{3}.
\end{displaymath}
In a similar way, the coefficient of $(xy)(xy)$ is 
\begin{displaymath}
	c_{x^2} +c_{x^2x}+c_{xx^2}+c_{x^2x^2} = -\frac{5}{24}.
\end{displaymath}
Finally, let us compute the coefficient of $x^my^n$ with $m,n\geq 1$. 
The BCH-cuts are $[x^my^n]$ and $(([x^i][x])\cdots [x])(([y^j][y])\cdots[y])$ $i=1,\dots, m$ and $j=1,\dots, n$. 
Hence, the coefficient is
\begin{displaymath}
	\frac{c_x}{m!n!} + \sum_{i=1,j=1}^{m,n} \frac{c_{x^ix^j}}{(m-i+1)!(n-j+1)!}.
\end{displaymath}
However, we observe that the coefficients of the monomials in the expansion of $\log_l(\exp_l(x))$ agree with those in $\log_l(\exp_l(x)\exp_l(y))$ (take $y=0$). 
The coefficient in  $\log_l(\exp_l(x))$ of $x^m x^n$ ($m,n \geq 1$) is 
\begin{eqnarray*}
	 \sum_{i=1,j=1}^{m,n} \frac{c_{x^ix^j}}{(m-i+1)!(n-j+1)!} &\text{if}& n\geq 2 \quad \text{or}\\
	 \frac{c_x}{(m+1)!}+\sum_{i=1,j=1}^{m,n} \frac{c_{x^ix^j}}{(m-i+1)!(n-j+1)!} &\text{if}& n =1.
\end{eqnarray*}
However, $\log_l(\exp_l(x)) = x$ so this coefficient is $0$. 
Therefore, the coefficient of $x^my^n$ in $\log_l(\exp_l(x)\exp_l(y))$ is
\begin{displaymath}
	\frac{1}{m!n!} \quad (n \geq 2) \quad\quad \text{or}\quad\quad \frac{m}{(m+1)!} \quad (n =1)
\end{displaymath}
\qed
\end{example}

%%%%%%%%%%%%%%%%%%%%%%%%%%%%%%%%%%%%%%%%%%%%%%%%%%%%
%    __       __                                  
%   /__\ ___ / _| ___ _ __ ___ _ __   ___ ___ ___ 
%  / \/// _ \ |_ / _ \ '__/ _ \ '_ \ / __/ _ \ __|
% / _  \  __/  _|  __/ | |  __/ | | | (__  __\__ \
% \/ \_/\___|_|  \___|_|  \___|_| |_|\___\___|___/
%

%%%%%%%%%%%%%%%%%%%%%%%%%%%%%%%%%%%%%%%%%%
% _____ _           _   _            _ _ 
%/__   \ |__   __ _| |_( )__    __ _| | |
%  / /\/ '_ \ / _` | __|/ __|  / _` | | |
% / /  | | | | (_| | |_ \__ \ | (_| | | |
% \/   |_| |_|\__,_|\__||___/  \__,_|_|_|
%                                        

\begin{bibdiv}
\begin{biblist}
%%%%%%%%%%%%%%%%%%%%%%%%%%%%%%%%%%%%%%%%%%%%%%%%%%%%

\bib{BlaCaOtRo09}{article}{
   author={Blanes, S.},
   author={Casas, F.},
   author={Oteo, J. A.},
   author={Ros, J.},
   title={The Magnus expansion and some of its applications},
   journal={Phys. Rep.},
   volume={470},
   date={2009},
   number={5-6},
   pages={151--238},
}

\bib{BonFul12}{book}{
   author={Bonfiglioli, A.},
   author={Fulci, R.},
   title={Topics in noncommutative algebra},
   series={Lecture Notes in Mathematics},
   volume={2034},
   note={The theorem of Campbell, Baker, Hausdorff and Dynkin},
   publisher={Springer, Heidelberg},
   date={2012},
   pages={xxii+539},
   isbn={978-3-642-22596-3},
}


\bib{ChaPa13}{article}{
   author={Chapoton, F.},
   author={Patras, F.},
   title={Enveloping algebras of preLie algebras, Solomon idempotents and
   the Magnus formula},
   journal={Internat. J. Algebra Comput.},
   volume={23},
   date={2013},
   number={4},
   pages={853--861},
}

\bib{DrGe04}{article}{
   author={Drensky, V.},
   author={Gerritzen, L.},
   title={Nonassociative exponential and logarithm},
   journal={J. Algebra},
   volume={272},
   date={2004},
   number={1},
   pages={311--320},
}

\bib{Dy47}{article}{
   author={Dynkin, E. B.},
   title={Calculation of the coefficients in the Campbell-Hausdorff formula},
   language={Russian},
   journal={Doklady Akad. Nauk SSSR (N.S.)},
   volume={57},
   date={1947},
   pages={323--326},
}

\bib{Fu00}{article}{
   author={Fuchs, P.},
   title={Bernoulli numbers and binary trees},
   note={Number theory (Liptovsk\'y J\'an, 1999)},
   journal={Tatra Mt. Math. Publ.},
   volume={20},
   date={2000},
   pages={111--117},
}

\bib{Ge04a}{article}{
   author={Gerritzen, L.},
   title={Planar rooted trees and non-associative exponential series},
   journal={Adv. in Appl. Math.},
   volume={33},
   date={2004},
   number={2},
   pages={342--365},
}


\bib{GeHo03}{article}{
   author={Gerritzen, L.},
   author={Holtkamp, R.},
   title={Hopf co-addition for free magma algebras and the non-associative
   Hausdorff series},
   journal={J. Algebra},
   volume={265},
   date={2003},
   number={1},
   pages={264--284},
}

\bib{MP10}{article}{
	title={Formal multiplications, bialgebras of distributions and nonassociative Lie theory},
	author={Mostovoy, J.},
	author={P{\'e}rez-Izquierdo, J. M.},
	journal={Transformation Groups},
	year={2010},
	volume={15},
	number={3},
	publisher={SP Birkhäuser Verlag Boston},
	pages={625-653},
}


\bib{MPS14}{article}{
	year={2014},
	journal={Bulletin of Mathematical Sciences},
	volume={4},
	number={1},
	title={Hopf algebras in non-associative Lie theory},
	publisher={Springer Basel},
	author={Mostovoy, J.},
	author={P{\'e}rez-Izquierdo, J. M.},
	author={Shestakov, I. P.},
	pages={129-173},
}

\bib{MPS14b}{article}{
	year={2014},
	journal={Journal of Algebra},
	volume={419},
	title={Nilpotent Sabinin algebras},
	author={Mostovoy, J.},
	author={P{\'e}rez-Izquierdo, J. M.},
	author={Shestakov, I. P.},
	pages={95-123},
	
	
}

\bib{SaMi87}{article}{
   author={Sabinin, L. V.},
   author={Mikheev, P. O.},
   title={Infinitesimal theory of local analytic loops},
   language={Russian},
   journal={Dokl. Akad. Nauk SSSR},
   volume={297},
   date={1987},
   number={4},
   pages={801--804},
   translation={
      journal={Soviet Math. Dokl.},
      volume={36},
      date={1988},
      number={3},
      pages={545--548},
   },
}


\bib{ShUm02}{article}{
   author={Shestakov, I. P.},
   author={Umirbaev, U. U.},
   title={Free Akivis algebras, primitive elements, and hyperalgebras},
   journal={J. Algebra},
   volume={250},
   date={2002},
   number={2},
   pages={533--548},
}

\bib{Sw69}{book}{
   author={Sweedler, M. E.},
   title={Hopf algebras},
   series={Mathematics Lecture Note Series},
   publisher={W. A. Benjamin, Inc., New York},
   date={1969},
   pages={vii+336},
}

\bib{Wi89}{article}{
   author={Wigner, D.},
   title={An identity in the free Lie algebra},
   journal={Proc. Amer. Math. Soc.},
   volume={106},
   date={1989},
   number={3},
   pages={639--640},
}


\bib{GW}{article}{
   author={Weingart, G.},
   title={On the axioms for Sabinin algebras},
   journal={J. Lie Theory},
  note={To appear}
}


\bib{Wo97}{article}{
   author={Woon, S. C.},
   title={A tree for generating Bernoulli numbers},
   journal={Math. Mag.},
   volume={70},
   date={1997},
   number={1},
   pages={51--56},
}
\end{biblist}
\end{bibdiv}
\end{document}